\documentclass[10pt,psamsfonts]{amsart}
\usepackage{amsmath}
\usepackage{amsthm}
\usepackage{amssymb}
\usepackage{amscd}
\usepackage{amsfonts}
\usepackage{amsbsy}
\usepackage{graphicx}
\usepackage[dvips]{psfrag}
\usepackage{array}
\usepackage{color}
\usepackage{epsfig}
\usepackage{url}

\newcolumntype{L}{>{\displaystyle}l}
\newcolumntype{C}{>{\displaystyle}c}
\newcolumntype{R}{>{\displaystyle}r}

\newcommand{\R}{\ensuremath{\mathbb{R}}}

\newcommand{\CC}{\mathcal{C}}
\newcommand{\CR}{\ensuremath{\mathcal{R}}}
\newcommand{\CF}{\ensuremath{\mathcal{F}}}
\newcommand{\CG}{\ensuremath{\mathcal{G}}}

\newcommand{\CM}{\ensuremath{\mathcal{M}}}
\newcommand{\CT}{\ensuremath{\mathcal{T}}}
\newcommand{\CO}{\ensuremath{\mathcal{O}}}
\newcommand{\ov}{\overline}

\newcommand{\T}{\theta}
\newcommand{\f}{\varphi}
\newcommand{\al}{\alpha}
\newcommand{\be}{\beta}
\newcommand{\cl}{\mbox{\normalfont{Cl}}}

\newcommand{\A}{\ensuremath{\mathcal{A}}}
\newcommand{\B}{\ensuremath{\mathcal{B}}}
\newcommand{\C}{\ensuremath{\mathcal{C}}}
\newcommand{\D}{\ensuremath{\mathcal{D}}}

\newcommand{\X}{\ensuremath{\mathcal{X}}}
\newcommand{\Lie}{\ensuremath{\mathcal{L}}}

\newcommand{\de}{\delta}

\def\p{\partial}
\def\e{\varepsilon}

\newtheorem {theorem} {Theorem} %[section]

\newtheorem {corollary} [theorem] {Corollary}
\newtheorem {lemma} [theorem] {Lemma}

\newtheorem {remark} {Remark}

\begin{document}

\title[]
{On the periodic solutions of a\\ generalized smooth or non-smooth\\ perturbed planar double pendulum}

\author[J. Llibre, D.D. Novaes and M.A. Teixeira]
{Jaume Llibre$^1$, Douglas D. Novaes$^2$  and Marco Antonio
Teixeira$^2$}

\address{$^1$ Departament de Matematiques,
Universitat Aut\`{o}noma de Barcelona, 08193 Bellaterra, Barcelona,
Catalonia, Spain} \email{jllibre@mat.uab.cat}

\address{$^2$ Departamento de Matematica, Universidade
Estadual de Campinas, Caixa Postal 6065, 13083--859, Campinas, SP,
Brazil} \email{ddnovaes@gmail.com, teixeira@ime.unicamp.br}

\let\thefootnote\relax\footnotetext{Corresponding author Douglas D. Novaes: Departamento de Matematica, Universidade
Estadual de Campinas, Caixa Postal 6065, 13083--859, Campinas, SP,
Brazil. Tel. +55 19 91939832, Fax. +55 19 35216094, email: ddnovaes@gmail.com}


\date{}
%\dedicatory{}

\maketitle

\begin{abstract}
We provide sufficient conditions for the existence of periodic
solutions with small amplitude of the non--linear planar double pendulum perturbed by smooth or non--smooth functions.
\end{abstract}

\section{Introduction and statement of the main results}\label{s1}

We consider a system of two point masses $m_1$ and $m_2$ moving in a
fixed plane, in which the distance between a point $P$ (called pivot)
and $m_1$ and the distance between $m_1$ and $m_2$ are fixed, and
equal to $l_1$ and $l_2$ respectively. We assume the masses do not
interact. We allow gravity to act on the masses $m_1$ and $m_2$.
This system is called the planar {\it double pendulum}.

\smallskip

The position of the double pendulum is determined by the two angles
$\phi_1$ and $\phi_2$ shown in Figure \ref{fig1}. The corresponding Lagrange equations of motion are
\begin{equation}\label{pendulum}
\begin{array}{L}
(m_{1}+m_{2})l_{1}\ddot{\phi}_{1}+m_{2}l_{2}\ddot{\phi}_{2}\cos(\phi_1-\phi_2)+
(m_{1}+m_{2})g\sin(\phi_1)\\
+m_2 l_2 \dot{\phi}_2^2\sin(\phi_1-\phi_2)=0,\\

m_2l_{1}\ddot{\phi}_{1}\cos(\phi_1-\phi_2)+m_2l_{2}\ddot{\phi}_{2}+m_2g\sin(\phi{2})+m_2l_1\dot{\phi}_1^2\sin(\phi_1-\phi_2)=0,
\end{array}
\end{equation}
where $g$ is the acceleration of the gravity.
For more details on these equations of motion see \cite{Ir}. Here
the dot denotes derivative with respect to the time $T$.

\begin{figure}[h]
\psfrag{P}{$P$}
\psfrag{A}{$l_{1}$} \psfrag{B}{$\phi_{1}$}
\psfrag{C}{$\phi_{2}$} \psfrag{D}{$l_{2}$} \psfrag{E}{$m_{1}$}
\psfrag{F}{$m_{2}$} \psfrag{G}{$g$}
\includegraphics[width=3cm]{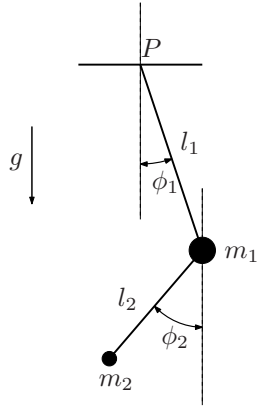}
\vskip 0cm \centerline{} \caption{\small \label{fig1} The planar
double pendulum.}
\end{figure}

\smallskip

The authors in \cite{LNT} have studied, in
the vicinity of the equilibrium $\phi_1= \phi_2= 0$, the persistence of periodic solutions of system (\ref{pendulum}) perturbed smoothly in the particular case when $m_{1}=m_{2}$ and $l_{1}=l_{2}$. Now $m_{1},m_{2},l_{1}$ and $l_{2}$ can take arbitrary positive values and we shall study the periodic orbits of system (\ref{pendulum}) which persist under smooth and non-smooth perturbations. 

\smallskip

Denote the expressions for $\ddot{\phi}_{1}$ and $\ddot{\phi}_{2}$ in \eqref{pendulum} respectively by $H_1(\phi_1,\dot{\phi}_1,\phi_2,\dot{\phi}_2)$ and $H_2(\phi_1,\dot{\phi}_1,\phi_2,\dot{\phi}_2)$. In this paper we shall consider the perturbed problem
\begin{equation}\label{pendulum2}
\begin{array}{LL}
\ddot{\phi}_1=&H_1(\phi_1,\dot{\phi}_1,\phi_2,\dot{\phi}_2)+\e \left(\hat F_1(t,\phi_1,\dot{\phi}_1,\phi_2,\dot{\phi}_2)+\hat F_2(t,\phi_1,\dot{\phi}_1,\phi_2,\dot{\phi}_2){\rm sgn}(\phi_{1})\right)\\
& +\e^2 \left(\hat R_1(t,\phi_1,\dot{\phi}_1,\phi_2,\dot{\phi}_2)+\hat R_2(t,\phi_1,\dot{\phi}_1,\phi_2,\dot{\phi}_2){\rm sgn}(\phi_{1})\right)+\CO(\e^3),\\
\ddot{\phi}_2=&H_2(\phi_1,\dot{\phi}_1,\phi_2,\dot{\phi}_2)+\e \left(\hat F_3(t,\phi_1,\dot{\phi}_1,\phi_2,\dot{\phi}_2)+\hat F_4(t,\phi_1,\dot{\phi}_1,\phi_2,\dot{\phi}_2){\rm sgn}(\phi_{2})\right)\\
& +\e^2 \left(\hat R_3(t,\phi_1,\dot{\phi}_1,\phi_2,\dot{\phi}_2)+\hat R_4(t,\phi_1,\dot{\phi}_1,\phi_2,\dot{\phi}_2){\rm sgn}(\phi_{2})\right)+\CO(\e^3).\\
\end{array}
\end{equation}

The function ${\rm sgn}(z)$ denotes the sign function, i.e.
\[
{\rm sgn}(z)=\left\{
\begin{array}{cl}
-1 & \mbox{if $z<0$,}\\
0 & \mbox{if $z=0$,}\\
1 & \mbox{if $z>0$.}
\end{array}
\right.
\]

Here the smooth functions $\hat F_i$ and $\hat R_i$ for $i=1,2,3,4$ define the perturbation. These functions are respectively $T_{F_i}$--periodic and $T_{R_i}$--periodic in
$t$ and respectively in resonance $p_{F_i}$:$q_{F_i}$ and $p_{R_i}$:$q_{R_i}$ with some of the periodic solutions
of the linearized unperturbed double pendulum, being $p$ and $q$ relatively prime positive integers for $p=p_{F_i},\,p_{R_i}$, $q=q_{F_i},\,q_{R_i}$ and $i=1,2,3,4$. We also assume that $F_i(t,0,0,0,0)=0$ for $i=1,2,3,4$. 

\begin{remark}\label{r1}
For simplicity, we can assume that the functions $\hat F_i$ and $\hat R_i$ for $i=1,2,3,4$ are $T$--periodic with $T=p T_j$ for some integer $p$ where $T_j$ for $j=1,2$ are the periods of the solutions of the linearized unperturbed double pendulum. Indeed, if we take $p$ the least common multiple among $p_{F_i}$ and $p_{R_i}$ for $i=1,2,3,4$, then there exists integers $n_{F_i}$ and $n_{R_i}$ such that $p=n_{F_i}\,p_{F_i}=n_{R_i}\,p_{R_i}$ for $i=1,2,3,4$. Hence
\[
pT_j=n_{F_i}q_{F_i}\dfrac{p_{F_i}}{q_{F_i}}T_j=n_{R_i}q_{R_i}\dfrac{p_{R_i}}{q_{R_i}}T_j.
\]
For $i=1,2,3,4$, and $j=1,2$.
\end{remark}

Note that the functions $\hat F_{i}$ and $\hat R_i$ for $i=1,2,3,4$, can be taken in a certain way arbitrary, i.e., only assuming some hypotheses. It makes us able to provide, in a physical context, real meaning for these functions. In our case, since we are working with discontinuity in the variables $\T_1$ and $\T_2$, the functions $\hat F_{1}$, $\hat F_{2}$, $\hat R_{1}$ and $\hat R_{2}$  could model the escapement for the particle $m_1$, and the functions $\hat F_{3}$, $\hat F_{4}$, $\hat R_{3}$ and $\hat R_{4}$ could model the escapement for the particle $m_2$. If we work with discontinuity in the variables $\T'_1$ and $\T'_2$, instead with discontinuity in the variables $\T_1$ and $\T_2$, the respective functions could model the Coulomb Friction. We also can work composing these two phenomena.  For more details on physical systems with discontinuous models see, for instance, \cite{AVK} and \cite{B}.

\smallskip

Now, we follow the steps:
\begin{itemize}
\item[(i)] proceed with the change of variable $\phi_1=\e\T_1$ and $\phi_2=\e\T_2$;
\[
\begin{array}{LL}
\ddot{\T}_1=&\dfrac{1}{\e}H_1(\e\T_1,\e\dot{\T}_1,\e\T_2, \e\dot{\T}_2)+\hat F_1(t,\e\T_1,\e\dot{\T}_1,\e\T_2, \e\dot{\T}_2)+\hat F_2(t,\e\T_1,\e\dot{\T}_1,\e\T_2, \e\dot{\T}_2){\rm sgn}(\T_{1})\\
&+\e\left(\hat R_1(t,\e\T_1,\e\dot{\T}_1,\e\T_2, \e\dot{\T}_2)+\hat R_2(t,\e\T_1,\e\dot{\T}_1,\e\T_2, \e\dot{\T}_2){\rm sgn}(\T_{1})\right)+\CO(\e^3),\\
\ddot{\T}_2=&\dfrac{1}{\e}H_2(\e\T_1,\e\dot{\T}_1,\e\T_2, \e\dot{\T}_2)+F_3(t,\e\T_1,\e\dot{\T}_1,\e\T_2, \e\dot{\T}_2)+F_4(t,\e\T_1,\e\dot{\T}_1,\e\T_2, \e\dot{\T}_2){\rm sgn}(\T_{2})\\
&+\e\left(\hat R_1(t,\e\T_1,\e\dot{\T}_1,\e\T_2, \e\dot{\T}_2)+\hat R_2(t,\e\T_1,\e\dot{\T}_1,\e\T_2, \e\dot{\T}_2){\rm sgn}(\T_{2})\right)+\CO(\e^3).\\
\end{array}
\]
\item[(ii)] expand in Taylor series, for $\e=0$, the expressions of $\ddot{\theta}_{1}$ and $\ddot{\theta}_{2}$;
\item[(iii)] Take a new time $t$ given by the rescaling $t =\al\,\tau$, with $\al=\sqrt{l_1\,m_1/(g\,m_2)}$; 
\item[(iv)] and finally, denote 
\[
\begin{array}{CCC}
a=\dfrac{m_1+m_2}{m_2}>1 &\textrm{and}& b=\dfrac{l_1(m_1+m_2)}{l_2m_2}>0.
\end{array}
\]
\end{itemize}

Hence, we obtain the following equations
of motion for the double pendulum
\begin{equation}\label{e1a}
\begin{array}{LL}
\T''_{1}=&-a\T_{1}+\T_{2}+\e\left(K_1(\tau)+F_1(\tau,\T_1,\T'_1,\T_2, \T'_2)+\left(K_2(\tau)+F_2(\tau,\T_1,\T'_1,\T_2, \T'_2)\right){\rm sgn}(\T_{1})\right)\\
&+\e^2 R_1(\tau,\T_1,\T'_1,\T_2,\T'_2,\e), \vspace{0.2cm}\\
\T''_{2}=&b\T_{1}-b\T_{2}+\e \left(K_3(\tau)+F_3(\tau,\T_1,\T'_1,\T_2, \T'_2)+\left(K_4(\tau)+F_4(\tau,\T_1,\T'_1,\T_2,\T'_2)\right){\rm sgn}(\T_{2})\right)\\
&+\e^2 R_2(\tau,\T_1,\T'_1,\T_2,\T'_2,\e),
\end{array}
\end{equation}
where now the prime denotes derivative with respect to the new time $t$. Here the functions $F_i$, for $i=1,2,3,4$, are linear in the spatial variables, and are given by
\[
F_i(\tau,\T_1,\T'_1,\T_2,\T'_2)=d_i^1(\tau)\T_1+d_i^2(\tau)\T'_1+d_i(\tau)^3\T_2+d_i(\tau)^4\T'_2.
\]
with
\[
\begin{array}{CC}
d_i^1(\tau)=\al^2\dfrac{\p \hat F_i}{\p\phi_1}(\al\tau,\vec{0}),&
d_i^2(\tau)=\al\dfrac{\p \hat F_i}{\p\dot{\phi}_1}(\al\tau,\vec{0}),\\
d_i^3(\tau)=\al^2\dfrac{\p \hat F_i}{\p\phi_2}(\al\tau,\vec{0}),&
d_i^4(\tau)=\al\dfrac{\p \hat F_i}{\p\dot{\phi}_2}(\al\tau,\vec{0}),
\end{array}
\]
and $K_i(\tau)=\al^2\hat R_i(\al\tau,0,0,0)$. Observe that, for $i,j=1,2,3,4$, $d_i^j(\tau)$ and $K_i(\tau)$ are $T/\al$--periodic functions.

\bigskip

The objective of this paper is to provide a system of equations whose simple zeros provide periodic solutions (see Figure \ref{PeriodicSolution}) of the
perturbed planar double pendulum \eqref{pendulum2}.

\begin{figure}[h]
\centering
\psfrag{Tp}{$\T'_1$}
\psfrag{T}{$\T_1$}
\psfrag{S}{$(\T_2,\T'_2)$}
\psfrag{C}{$m$}
\includegraphics[width=8.5cm]{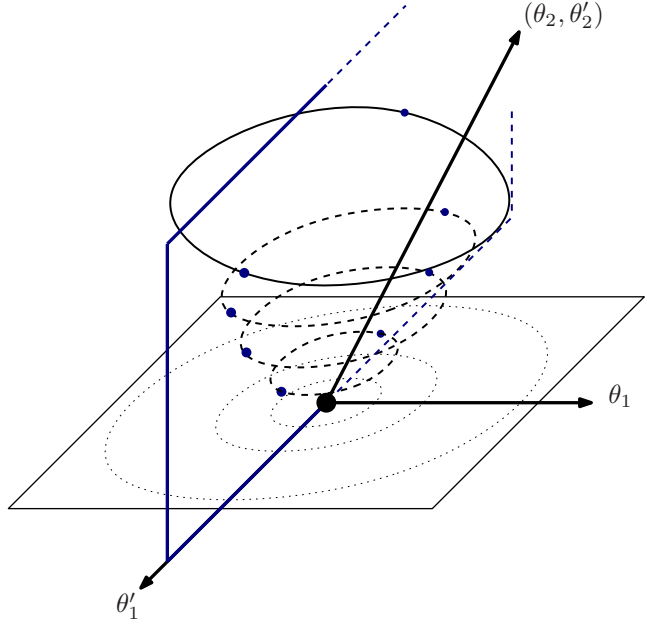}
\vskip 0cm \centerline{} \caption{\small \label{PeriodicSolution} Periodic solution of the perturbed system \eqref{pendulum2} converging to the origin, when $\e\to 0$.}
\end{figure}

\smallskip

In order to present our results we need
some preliminary definitions and notations.

\smallskip

The unperturbed system \eqref{e1a} has a unique singular point, the
origin with eigenvalues $\pm \omega_{1} \, i, \pm \omega_{2}\, i$, where
\[
\omega_{1}= \dfrac{\sqrt{a+b-\sqrt{\Delta}}}{\sqrt{2}}, \quad \omega_{2}=\dfrac{\sqrt{a+b+\sqrt{\Delta}}}{\sqrt{2}},
\]
with $\Delta=(a-b)^2+4 b>0$. Consequently this system in the phase space $(\T_1,\T'_1,
\T_2,\T'_2)$ has two planes filled with periodic solutions except the origin. The periods of such periodic orbits are
\[
T_1= \dfrac{2\pi}{\omega_{1}} \quad \mbox{or} \quad T_2=
\dfrac{2\pi}{\omega_{2}}.
\]
These periodic orbits live into the planes associated to the eigenvectors
with eigenvalues $\pm \omega_{1} \, i$ or $\pm \omega_{2}\, i$, respectively. We shall study which of these
periodic solutions persist for the perturbed system \eqref{pendulum2} when
the parameter $\e$ is sufficiently small and the functions of perturbation
$\hat F_i$ and $\hat R_i$ for $i=1,2,3,4$ have period either $p \al T_1$, or $p\al T_2$, with $p$ positive integer.

\begin{remark}\label{CH}
We say that the {\it Crossing Hypothesis} is satisfied if there exists a compact set $D\subset\R^4$ such that every orbit starting in $D$ reaches the set of discontinuity only at its crossing regions (see Appendix A).
\end{remark}

Let $X_{X_0,Y_0}{(\tau)}$ be the periodic function
\[
X_{X_0,Y_0}{(\tau)}= Y_0 \cos \left(\omega_{1}\, \tau \right)+X_0 \sin
\left(\omega_{1}\, \tau \right),
\]
then we define the non--smooth function $\CF_1(X_0,Y_0)$ by
\begin{equation}\label{e3a}
\begin{array}{L}
\int_0^{pT_1}\sin\left(\omega_{1} \, \tau\right)\left(2 b (\bar{F}_1+K_1(\tau))+(\bar{F}_3+K_3(\tau)) \left(a-b+\sqrt{\Delta}\right)\right)\,d\tau\\
+\int_0^{pT_1}\sin\left(\omega_{1} \, \tau\right)\left(2 b (\bar{F}_2+K_2(\tau))+(\bar{F}_4+K_4(\tau)) \left(a-b+\sqrt{\Delta}\right)\right)\,{\rm sgn}(X_{X_0,Y_0}{(\tau)}) \,d\tau,
\end{array}
\end{equation}
and the non--smooth function $\CF_2(X_0,Y_0)$ by
\begin{equation}\label{e3b}
\begin{array}{L}
\int_0^{pT_1}\cos\left(\omega_{1} \, \tau\right)\left(2 b (\bar{F}_1+K_1(\tau))+(\bar{F}_3+K_3(\tau)) \left(a-b+\sqrt{D}\right)\right)\,d\tau\\
+\int_0^{pT_1}\cos\left(\omega_{1} \, \tau\right)\left(2 b (\bar{F}_2+K_2(\tau))+(\bar{F}_4+K_4(\tau)) \left(a-b+\sqrt{D}\right)\right)\,{\rm sgn}(X_{X_0,Y_0}{(\tau)}) \,d\tau.
\end{array}
\end{equation}
where
\begin{equation}\label{jj}
\bar F_i= F_i(\tau,A_1,B_1,C_1,D_1)
\end{equation}
for $i=1,2,3,4$ with
\[
\begin{array}{l}
A_1= \dfrac{\left(-a+b+\sqrt{\Delta}\right)}{2\, b\,\omega_{1}} \left(X_0 \cos \left(\omega_{1}\, \tau \right) + Y_0 \sin\left(\omega_{1}\, \tau \right)\right), \vspace{0.2cm}\\
B_1= \dfrac{ \left(-a+b+\sqrt{\Delta}\right)}{2 b} \left(Y_0 \cos \left(\omega_{1}\, \tau \right)-X_0 \sin
\left(\omega_{1}\, \tau \right)\right), \vspace{0.2cm}\\
C_1= \dfrac{1}{\omega_{1}} \left(X_0 \cos \left(\omega_{1}\, \tau \right) + Y_0 \sin
\left(\omega_{1}\, \tau \right)\right), \vspace{0.2cm}\\
D_1= Y_0 \cos \left(\omega_{1}\, \tau \right)-X_0 \sin
\left(\omega_{1}\, \tau \right).
\end{array}
\]

A zero $(X_0^*,Y_0^*)$ of the system of the non-smooth functions
\begin{equation}\label{e4}
\CF_1(X_0,Y_0)=0,\quad \CF_2(X_0,Y_0)=0,
\end{equation}
such that
\[
\det \left(\left.
\dfrac{\p(\CF_1,\CF_2)}{\p(X_0,Y_0)}\right|_{(X_0,Y_0)=
(X_0^*,Y_0^*)}\right) \neq 0,
\]
is called a {\it simple zero} of system \eqref{e4}.

\smallskip

Our main result on the periodic solutions of the non-smooth perturbed double
pendulum \eqref{pendulum2} which bifurcate from the periodic solutions of
the unperturbed double pendulum \eqref{pendulum} with period $T_1$
traveled $p$ times is the following.

\begin{theorem}\label{t1}
Assume that the functions $\hat F_i$ and $\hat R_i$ of the non-smooth perturbed double pendulum \eqref{pendulum2} are periodic in $t$ of period
$p\al T_1$ with $p$ positive integer. Also assume that the Crossing Hypothesis (see Remark \ref{CH}) is satisfied. Then
for $|\e|>0$ sufficiently small and for every simple zero
$(X_0^*,Y_0^*)\neq (0,0)$ of the non-smooth system \eqref{e4} such that the orbits pass by $D$, the
non-smooth perturbed double pendulum \eqref{pendulum2} has a $p\al T_1$--periodic solution $(\phi_1(t,\e),\phi_2(t,\e))\to(0,0)$ when $\e\to 0$.
\end{theorem}

Theorem \ref{t1} is proved in section \ref{s3}. Its proof is based
in the averaging theory for computing periodic solutions, see the
Appendix B.

\smallskip

Note that the periodic solution given in Theorem \ref{t1} is a periodic solution bifurcating at $\e=0$ from the equilibrium of system \eqref{pendulum2} localized at the origin of coordinates. For $|\e|>0$ sufficiently small this orbits is close to the plane defined by the eigenvectors of the eigenvalues $\pm i\omega_1$.

\smallskip

We provide an application of Theorem \ref{t1} in the following
corollary, which will be proved in section \ref{s4}.

\begin{corollary}\label{c1}
Suppose that $F_1=y/\al^2+f_1$, $F_3=w/\al^2+f_3$, and $f_1$, $F_2$, $f_3$, and $F_4$ has no linear term. Also suppose that $R_2=1/\al^2+r_2$, $R_4=1/\al^2+r_4$ and $R_1$, $r_2$, $R_3$, and $r_4$ has no constant term. Moreover, assume that all functions are $p\al T_1/q$--periodic. Then the differential system \eqref{pendulum2} for $|\e|> 0$ sufficiently small has one $p\al T_1$--periodic solution $(\phi_1(t,\e),\phi_2(t,\e))\to(0,0)$ when $\e\to 0$.
\end{corollary}

Now let $Z^{Z_0,W_0}{(\tau)}$ be the periodic function
\[
Z^{Z_0,W_0}{(\tau)}= W_0 \cos \left(\omega_{2}\, \tau\right)+Z_0 \sin
\left(\omega_{2}\, \tau \right),
\]
then we define the non--smooth function $\CF^1(Z_0,W_0)$ by
\begin{equation}\label{e4a}
\begin{array}{L}
\int_0^{pT_2}\sin\left(\omega_{2} \, \tau\right)\left(-2 b (\bar{F}_1+K_1(\tau))+(\bar{F}_3+K_3(\tau)) \left(-a+b+\sqrt{\Delta}\right)\right)\,d\tau\\
+\int_0^{pT_2}\sin\left(\omega_{2} \, \tau\right)\left(2 b (\bar{F}_2+K_2(\tau))+(\bar{F}_4+K_4(\tau)) \left(-a+b+\sqrt{\Delta}\right)\right)\,{\rm sgn}(Z^{Z_0,W_0}{(\tau)}) \,d\tau,
\end{array}
\end{equation}
and the non--smooth function $\CF^2(Z_0,W_0)$ by
\begin{equation}\label{e4b}
\begin{array}{L}
\int_0^{pT_2}\cos\left(\omega_{2} \, \tau\right)\left(-2 b (\bar{F}_1+K_1(\tau))+(\bar{F}_3+K_3(\tau)) \left(a-b+\sqrt{D}\right)\right)\,d\tau\\
+\int_0^{pT_2}\cos\left(\omega_{2} \, \tau\right)\left(2 b (\bar{F}_2+K_2(\tau))+(\bar{F}_4+K_4(\tau)) \left(a-b+\sqrt{D}\right)\right)\,{\rm sgn}(Z^{Z_0,W_0}{(\tau)}) \,d\tau.
\end{array}
\end{equation}
where
\[
\bar F_i= F_i(t,A_2,B_2,C_2,D_2)
\]
for $i=1,2,3,4$ with
\[
\begin{array}{l}
A_2= -\dfrac{\left(a-b+\sqrt{\Delta}\right)}{2\, b\,\omega_{2}} \left(Z_0 \cos \left(\omega_{2}\, t \right) + W_0 \sin\left(\omega_{2}\, t \right)\right), \vspace{0.2cm}\\
B_2= -\dfrac{ \left(a-b+\sqrt{\Delta}\right)}{2 b} \left(W_0 \cos \left(\omega_{2}\, t \right)-Z_0 \sin
\left(\omega_{2}\, t \right)\right), \vspace{0.2cm}\\
C_2= \dfrac{1}{\omega_{2}} \left(Z_0 \cos \left(\omega_{2}\, t \right) + W_0 \sin
\left(\omega_{2}\, t \right)\right), \vspace{0.2cm}\\
D_2= W_0 \cos \left(\omega_{2}\, t \right)-Z_0 \sin
\left(\omega_{2}\, t \right).
\end{array}
\]

Consider the non-linear and non-smooth system
\begin{equation}\label{e4a}
\CF^1(Z_0,W_0)=0,\quad \CF^2(Z_0,W_0)=0.
\end{equation}

\smallskip

Our main result on the periodic solutions of the non-smooth perturbed double
pendulum \eqref{pendulum2} which bifurcate from the periodic solutions of
the unperturbed double pendulum \eqref{e1a} with period $T_2$
traveled $p$ times is the following.

\begin{theorem}\label{t1a}
Assume that the functions $\hat F_i$ and $\hat R_i$ of the non-smooth perturbed double pendulum \eqref{pendulum2} are periodic in $t$ of period
$p\al T_2$ with $p$ positive integer. Also assume that the Crossing Hypothesis (see Remark \ref{CH}) is satisfied. Then
for $\e>0$ sufficiently small and for every simple zero
$(Z_0^*,W_0^*)\neq (0,0)$ of the non-smooth system \eqref{e4a} such that the orbits pass by $D$, the
non-smooth perturbed double pendulum \eqref{pendulum2} has a $p\al T_2$--periodic solution $(\phi_1(t,\e),\phi_2(t,\e))\to(0,0)$ when $\e\to 0$.

\end{theorem}

Theorem \ref{t1a} is also proved in section \ref{s3}.

\smallskip

Again the periodic solution given in Theorem \ref{t1a} is a periodic solution bifurcating at $\e=0$ from the equilibrium of system \eqref{pendulum2} localized at the origin of coordinates. For $|\e|>0$ sufficiently small this orbits is close to the plane defined by the eigenvectors of the eigenvalues $\pm i\omega_2$.

\smallskip

We provide an application of Theorem \ref{t1a} in the following
corollary, which will be proved in section \ref{s4}.

\begin{corollary}\label{c2}
Suppose that $F_1=y/\al^2+f_1$, $F_3=w/\al^2+f_3$, and $f_1$, $F_2$, $f_3$, and $F_4$ has no linear term. Also suppose that $R_2=1/\al^2+r_2$, $R_4=1/\al^2+r_4$ and $R_1$, $r_2$, $R_3$, and $r_4$ has no constant term. Moreover, assume that all functions are $p\al T_2/q$--periodic. Then the differential system \eqref{pendulum2} for $|\e|> 0$
sufficiently small has one $p\al T_2$--periodic solution $(\phi_1(t,\e),\phi_2(t,\e))\to(0,0)$ when $\e\to 0$.
\end{corollary}

\section{Proofs of Theorems \ref{t1} and \ref{t1a}}\label{s3}

Introducing the variables $(x,y,z,w)= (\T_1,\T'_1,\T_2,\T'_2)$ we
write the differential system of the non-smooth perturbed double pendulum
\eqref{e1a} as a first--order differential system defined in $\R^4$.
Thus we have the differential system
\begin{equation}\label{d1}
\begin{array}{LL}
x' =& y,\\
y' =& -ax+z+\e \left(K_1(\tau)+F_1(\tau,x,y,z,w)+(K_2(\tau)+F_2(\tau,x,y,z,w))\,{\rm sgn}(x)\right)\\
&+\e^2\,R_1(\tau,x,y,z,w,\e),\\
z' =& w,\\
w' =& bx-bz+\e \left(K_3(\tau)+F_3(\tau,x,y,z,w)+(K_4(\tau)+F_4(\tau,x,y,z,w))\,{\rm sgn}(z)\right)\\
&+\e^2\,R_2(\tau,x,y,z,w,\e).
\end{array}
\end{equation}

System \eqref{d1} with $\e=0$ is equivalent to the unperturbed
double pendulum system \eqref{e1a}, called in what follows simply by
the {\it unperturbed system}. Otherwise we have the {\it perturbed
system}.

\smallskip

Instead of working with the discontinuous differential system
\eqref{d1} we shall work with the smooth differential system
\begin{equation}\label{d1a}
\begin{array}{LL}
x' =& y,\\
y' = &-ax+z+\e \left(K_1(\tau)+F_1(\tau,x,y,z,w)+(K_2(\tau)+F_2(\tau,x,y,z,w))\,s_{\de}(y)\right)\\
&+\e\,R_1(\tau,x,y,z,w,\e),\\
z' = &w,\\
w' = &bx-bz+\e \left(K_3(\tau)+F_3(\tau,x,y,z,w)+(K_4(\tau)+F_4(\tau,x,y,z,w))\,s_{\de}(w)\right)\\
&+\e\,R_2(\tau,x,y,z,w,\e).
\end{array}
\end{equation}
where $s_{\de}(x)$ is the smooth function defined in Figure
\ref{fig2}, such that
\[
\lim_{\de \to 0} s_{\de}(x)= {\rm sgn}(x).
\]

\begin{figure}[h]
\psfrag{A}{$-1$} \psfrag{B}{$1$} \psfrag{D}{$\delta$}
\psfrag{E}{$-\delta$} \psfrag{I}{sign$(x)$}
\psfrag{J}{$s_{\delta}(x)$} \psfrag{X}{$x$}
\centerline{\includegraphics[width=10cm]{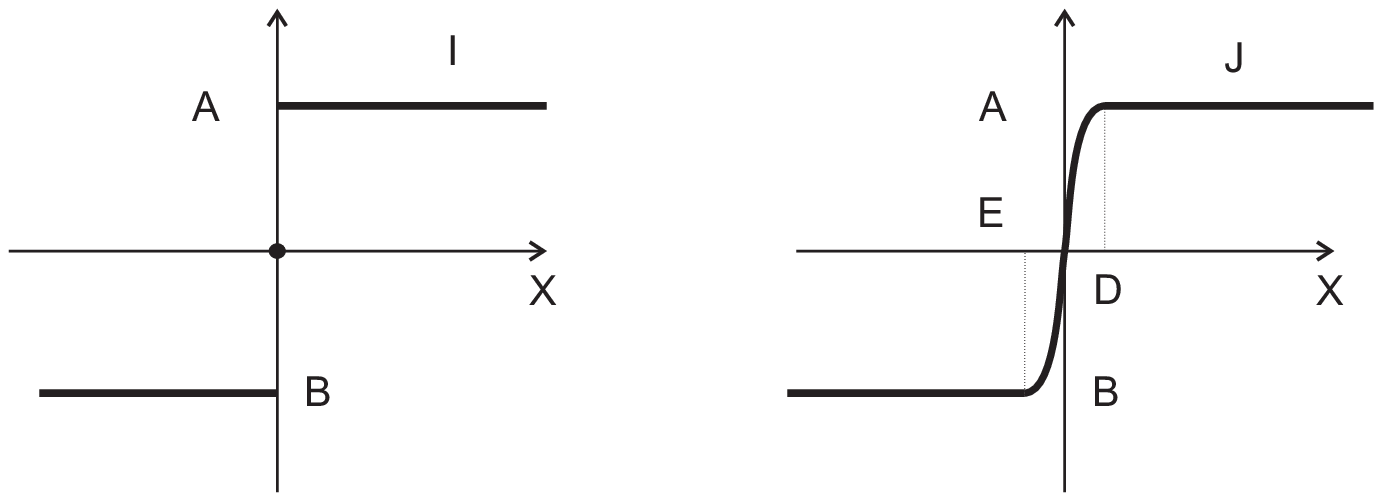}} \vskip 0cm
\centerline{} \caption{\small \label{fig2} The functions sign$(x)$
and $s_{\delta}(x)$.}
\end{figure}

\smallskip

We shall write system \eqref{d1a} in such a way that the linear part
at the origin of the unperturbed system will be in its real normal Jordan form. Then, doing
the change of variables $(\tau,x,y,z,w)\to (\tau,X,Y,Z,W)$ given by
\begin{equation}\label{nn}
\left(
\begin{array}{c}
X\\
Y\\
Z\\
W
\end{array}
\right)= \left(
\begin{array}{cccc}
 \dfrac{b \omega_{1}}{\sqrt{\Delta}} & 0 & \dfrac{\omega_{1} \left(a-b+\sqrt{\Delta}\right)}{2 \sqrt{\Delta}} & 0 \\
 0 & \dfrac{b}{\sqrt{\Delta}} & 0 & \dfrac{a-b+\sqrt{\Delta}}{2 \sqrt{\Delta}} \\
 -\dfrac{b \omega_{2}}{\sqrt{\Delta}} & 0 & \dfrac{\omega_{2}\left(-a+b+\sqrt{\Delta}\right)}{2 \sqrt{\Delta}} & 0 \\
 0 & -\dfrac{b}{\sqrt{\Delta}} & 0 & \dfrac{-a+b+\sqrt{\Delta}}{2 \sqrt{\Delta}}
\end{array}
\right)\left(
\begin{array}{c}
x\\
y\\
z\\
w
\end{array}
\right),
\end{equation}
the differential system \eqref{d1a} becomes
\begin{equation}\label{d2}
\begin{array}{LL}
X' =& \omega_{1} Y, \vspace{0.2cm}\\
Y' =& -\omega_{1} X+ \e \dfrac{1}{2\sqrt{\Delta}} \left(2b\left(K_1(\tau)+\tilde{F}_{1}+(K_2(\tau)+\tilde{F}_{2})\, s_{\de}(\A)\right)\right)\\
&+\e\dfrac{1}{2\sqrt{\Delta}} \left(\left(a-b+\sqrt{\Delta}\right)\left(K_3(\tau)+\tilde{F}_{3}+(K_4(\tau)+\tilde{F}_{4})\, s_{\de}(\C)\right) \right)+\e^2\tilde{R}_1,\\
Z' =& \omega_{2} W,\vspace{0.2cm}\\
W' =& -\omega_{2}Z+\e \dfrac{1}{2\sqrt{\Delta}} \left(-2b\left(K_1(\tau)+\tilde{F}_{1}+(K_2(\tau)+\tilde{F}_{2})\, s_{\de}(\A)\right)\right)\\
&+\e\dfrac{1}{2\sqrt{\Delta}}\left(\left(-a+b+\sqrt{\Delta}\right)\left(K_3(\tau)+\tilde{F}_{3}+(K_4(\tau)+\tilde{F}_{4})\, s_{\de}(\C)\right) \right)+\e^2\tilde{R}_2,
\end{array}
\end{equation}
where $\tilde F_i(t,X,Y,Z,W)=F_i(t,\A,\B,\C,\D)$ for $i=1,2,3,4$
with
\[
\begin{array}{l}
\A= \dfrac{\left(-a+b+\sqrt{\Delta}\right)}{2\, b\, \omega_{1}}X -\dfrac{ \left(a-b+\sqrt{\Delta}\right)}{2\, b\, \omega_{2}}Z,\vspace{0.2cm}\\
\B= \dfrac{ \left(-a+b+\sqrt{\Delta}\right)}{2 b}Y-\dfrac{ \left(a-b+\sqrt{\Delta}\right)}{2 b}W,\vspace{0.2cm}\\
\C= \dfrac{1}{\omega_{1}}X+\dfrac{1}{\omega_{2}}Z,\vspace{0.2cm}\\
\D= Y+W,
\end{array}
\]
and $\tilde{R}_i=\tilde{R}_i(X,Y,Z,W,\e)$.
Note that the linear part of the differential system \eqref{d2} at
the origin is in its real normal Jordan form.

\begin{lemma}\label{L1}
The periodic solutions of the differential system \eqref{d2} with
$\e=0$ are
\begin{equation}\label{d3}
\begin{array}{l}
X_{X_0,Y_0}{(\tau)}= X_0 \cos \left(\omega_{1}\, \tau \right) + Y_0 \sin
\left(\omega_{1}\, \tau \right),\\
Y_{X_0,Y_0}{(\tau)}= Y_0 \cos \left(\omega_{1}\, \tau \right)-X_0 \sin
\left(\omega_{1}\, \tau \right),\\
Z_{X_0,Y_0}{(\tau)}= 0,\\
W_{X_0,Y_0}{(\tau)}= 0,
\end{array}
\end{equation}
of period $T_1$, and
\begin{equation}\label{d4}
\begin{array}{l}
X^{Z_0,W_0}{(\tau)}= 0,\\
Y^{Z_0,W_0}{(\tau)}= 0,\\
Z^{Z_0,W_0}{(\tau)}= Z_0 \cos \left(\omega_{2}\, \tau \right)+W_0 \sin
\left(\omega_{2}\, \tau \right),\\
W^{Z_0,W_0}{(\tau)}= W_0 \cos \left(\omega_{2}\, \tau\right)-Z_0 \sin
\left(\omega_{2}\, \tau \right),
\end{array}
\end{equation}
of period $T_2$.
\end{lemma}

\begin{proof}
Since system \eqref{d2} with $\e=0$ is a linear differential system,
the proof follows easily.
\end{proof}

\begin{proof}[Proof of Theorem \ref{t1}]
Assume that the functions $\hat F_i$ and $\hat R_i$ of the non-smooth perturbed double pendulum
with equations of motion \eqref{e1a} are periodic in $t$ of period
$p\al T_1$ with $p$ positive integer. Thus $K_i$ and $F_i$ are $pT_1$--periodic functions, i.e., the differential system \eqref{e1a} and the
periodic solutions \eqref{d3} have the same period $pT_1$.

\smallskip

It is well known that a Poincar\'{e} map defined in a smooth differential system is smooth. So the Poincar\'{e} maps associated to the periodic orbits of the differential system \eqref{d1a} are smooth. 

\smallskip

The Poincar\'{e} maps, restricted at $D$, associated to the periodic solutions of the non-smooth differential system \eqref{d1}, which are perturbations of the periodic solutions \eqref{d3} are also smooth. Since the orbits starting in $D$ reaches the discontinuity set only at the of crossing region (see Appendix A), such Poincar\'{e} maps are compositions of smooth Poincar\'{e} maps. In a similar way it follows that the Poincar\'{e} maps, restricted at $D$, associated to the periodic solutions of the non-smooth differential system \eqref{d1}, which are perturbations of the periodic solutions \eqref{d4} are also smooth.

\smallskip

We can use Theorem \ref{tt} (see the Appendix B) for computing some periodic solutions of the smooth systems. The periodic solutions are zeros of the displacement function, which is the Poincar\'{e} map associated to periodic solutions minus the identity. In fact, the non-linear function \eqref{eq:7} whose zeros can provide periodic solutions, is the first term of order $\e$ of the displacement function. See for more details the proof of Theorem \ref{tt} in \cite{BFL}.

\smallskip

Since the Poincar\'{e} maps associated to periodic solutions of system \eqref{d1}, coming from the perturbed periodic solutions \eqref{d3} or \eqref{d4}, are smooth and these Poincar\'{e} maps are the limit of the Poincar\'{e} maps associated to the smooth system \eqref{d1a}, for which we can use Theorem \ref{tt}, it follows that we also can use Theorem \ref{tt} for computing some of the periodic solutions of the non-smooth system \eqref{d1}. In other words, we can apply Theorem \ref{tt} to the smooth systems \eqref{d1a} and then pass to the limit, when $\de\to0$, the function \eqref{eq:7} for obtaining a function whose zeros can give periodic solutions of the non-smooth system \eqref{d1}.

\smallskip

We shall apply Theorem \ref{tt} of the Appendix B to differential
system \eqref{d2}. We note that system \eqref{d2} can be written as
system \eqref{eq:4} taking
\[
{\bf x}=\left(
\begin{array}{c}
X\vspace{0.2cm}\\
Y\vspace{0.2cm}\\
Z\vspace{0.2cm}\\
W
\end{array}
\right), \quad t=\tau, \quad G_0(t,{\bf x})=\left(
\begin{array}{c}
 \omega_{1}\, Y,\vspace{0.2cm}\\
- \omega_{1}\, X,\vspace{0.2cm}\\
 \omega_{2}\, W,\vspace{0.2cm}\\
- \omega_{2}\, Z
\end{array}
\right),
\]
\[
G_1(t,{\bf x})=\left(
\begin{array}{c}
0 \vspace{0.2cm}\\
\dfrac{b}{\sqrt{\Delta}}\left(K_1(\tau)+\tilde{F}_{1}+(K_2(\tau)+\tilde{F}_{2})\,s_{\de}(\A)\right)\\+\dfrac{a-b+\sqrt{\Delta}}{2\sqrt{\Delta}}\left(K_3(\tau)+\tilde{F}_{3}+(K_4(\tau)+\tilde{F}_{4})\,s_{\de}(\C)\right)\vspace{0.2cm}\\
0 \vspace{0.2cm}\\
-\dfrac{b}{\sqrt{\Delta}}\left(K_1(\tau)+\tilde{F}_{1}+(K_2(\tau)+\tilde{F}_{2})\,s_{\de}(\A)\right)\\+\dfrac{-a+b+\sqrt{\Delta}}{2\sqrt{\Delta}}\left(K_3(\tau)+\tilde{F}_{3}+(K_4(\tau)+\tilde{F}_{4})\,s_{\de}(\C)\right) \end{array}
\right)
\]
\[
\mbox{and} \quad G_2(t,{\bf x,\e})=\left(
\begin{array}{c}
0 \vspace{0.2cm}\\
R_1 \vspace{0.2cm}\\
0 \vspace{0.2cm}\\
R_2
\end{array}
\right).
\]

We shall study which periodic solutions \eqref{d3} of the
unperturbed system \eqref{d2} with $\e=0$ can be continued to
periodic solutions of the perturbed system \eqref{d2} for $\e\neq
0$ sufficiently small.

\smallskip

We shall describe the different elements which appear in the
statement of Theorem \ref{tt} in the particular case of the
differential system \eqref{d2}. Thus we have that $\Omega= \R^4$,
$k=2$ and $n=4$. Let $r_1>0$ be arbitrarily small and let $r_2>0$ be
arbitrarily large. We take the open and bounded subset $V$ of the
plane $Z=W=0$ as
\[
V= \{ (X_0,Y_0,0,0)\in \R^4 : r_1< \sqrt{X_0^2+Y_0^2} < r_2 \}.
\]
As usual $\cl(V)$ denotes the closure of $V$. If $\al=(X_0,Y_0)$,
then we can identify $V$ with the set
\[
\{ \al\in \R^2 : r_1< ||\al|| < r_2 \},
\]
here $|| \cdot ||$ denotes the Euclidean norm of $\R^2$. The
function $\be: \cl(V)\to \R^2$ is $\be(\al)= (0,0)$. Therefore, in
our case the set
\[
\mathcal{Z}=\left\{ {\bf z}_{\alpha}=\left( \alpha,
\beta(\alpha)\right),~~\alpha\in \cl(V) \right\}= \{
(X_0,Y_0,0,0)\in \R^4 : r_1\leq \sqrt{X_0^2+Y_0^2} \leq r_2 \}.
\]
Clearly for each ${\bf z}_{\alpha}\in \mathcal{Z}$ we can consider
the periodic solution ${\bf x}(\tau, {\bf z}_{\alpha})= (
X(\tau),Y(\tau)$, $0,0)$ given by \eqref{d3} with period $pT_{1}$.

\smallskip

Computing the fundamental matrix $M_{{\bf z}_{\alpha}}(\tau)$ of the
linear differential system \eqref{d2} with $\e=0$ associated to the
$T$--periodic solution ${\bf z}_{\alpha}= (X_0,Y_0,0,0)$ such that
$M_{{\bf z}_{\alpha}}(0)$ be the identity of $\R^4$, we get that
$M(\tau)= M_{{\bf z}_{\alpha}}(\tau)$ is equal to
\[
\left(
\begin{array}{cccc}
\cos \left(\omega_{1}\, \tau\right) & \sin
\left(\omega_{1}\, \tau\right) & 0 & 0 \\
-\sin \left(\omega_{1}\, \tau\right) & \cos
\left(\omega_{1}\, \tau\right) & 0 & 0 \\
0 & 0 & \cos \left(\omega_{2}\, \tau\right) & \sin
\left(\omega_{2}\, \tau\right) \\
0 & 0 & -\sin \left(\omega_{2}\, \tau\right) & \cos
\left(\omega_{2}\, \tau\right)
\end{array}
\right).
\]
Note that the matrix $M_{{\bf z}_{\alpha}}(\tau)$ does not depend on
the particular periodic solution ${\bf x}(\tau, {\bf z}_{\alpha})$.
Since the matrix
\[
M^{-1}(0)-M^{-1}(pT_1)=\left(
\begin{array}{cccc}
 0 & 0 & 0 & 0 \\
 0 & 0 & 0 & 0 \\
 0 & 0 & 2 \sin ^2\left(\dfrac{p\pi\,\omega_{2} }{\omega_{1}}\right) & \sin \left(\dfrac{2p\pi\,\omega_{2} }{\omega_{1}}\right) \\
 0 & 0 & -\sin \left(\dfrac{2p\pi\,\omega_{2} }{\omega_{1}}\right) & 2 \sin ^2\left(\dfrac{p\pi\,\omega_{2} }{\omega_{1}}\right)
\end{array}\right),
\]
satisfies the assumptions of statement (ii) of Theorem \ref{tt}
because the determinant
\[
\left|
\begin{array}{ll}
2 \sin ^2\left(\dfrac{p\pi\,\omega_{2} }{\omega_{1}}\right) & \sin \left(\dfrac{2p\pi\,\omega_{2} }{\omega_{1}}\right) \\
 -\sin \left(\dfrac{2p\pi\,\omega_{2} }{\omega_{1}}\right) & 2 \sin ^2\left(\dfrac{p\pi\,\omega_{2} }{\omega_{1}}\right)\end{array}
\right|= 4 \sin ^2\left(\dfrac{p\pi\,\omega_{2} }{\omega_{1}}\right)\neq 0.
\]
So we can apply Theorem \ref{tt} to system \eqref{d2}.

\smallskip

Now $\xi: \R^4\to \R^2$ is $\xi(X,Y,Z,W)= (X,Y)$. We calculate, when $\de\to\ 0$, the
function
\[
\CG(X_0,Y_0)=\CG(\alpha)=\xi\left( \dfrac{1}{pT_{1}} \displaystyle \int
_0^{pT_1} M_{{\bf z}_{\alpha}}^{-1}(t)G_1(t,{\bf x}(t,{\bf
z}_{\alpha})) dt\right),
\]
and we obtain the function $\CG_2(X_0,Y_0)$
\begin{equation}\label{S1}
\begin{array}{L}
-\dfrac{1}{2\sqrt{\Delta}pT_{1} } \int_0^{pT_1}\left[
\sin\left(\omega_{1} \, \tau\right) \left(2b\left(K_1(\tau)+\bar{F}_{1}+(K_2(\tau)+\bar{F}_{2})\,{\rm sgn}\left(\dfrac{ \-a+b+\sqrt{\Delta}}{2 b}X_{X_0,Y_0}{(\tau)}\right)\right)\right.\right.\\
\left.\left.+\left(a-b+\sqrt{\Delta}\right)\left(K_3(\tau)+\bar{F}_{3}+(K_4(\tau)+\bar{F}_{4})\,{\rm sgn}(X_{X_0,Y_0}{(\tau)})\right)\right)\right] d\tau, \vspace{0.2cm}\\
\end{array}
\end{equation}
and the function $\CG_2(X_0,Y_0)$
\begin{equation}\label{S1a}
\begin{array}{L}
\dfrac{1}{2\sqrt{\Delta}pT_{1}} \int_0^{pT_1}\left[\cos
\left(\omega_{1} \, \tau\right) \left(+2b\left(K_1(\tau)+\bar{F}_{1}+(K_2(\tau)+\bar{F}_{2})\,{\rm sgn}\left(\dfrac{ \-a+b+\sqrt{\Delta}}{2 b}X_{X_0,Y_0}{(\tau)}\right)\right)\right.\right.\\
+\left.\left.\left(a-b+\sqrt{\Delta}\right)\left(K_3(\tau)+\bar{F}_{3}+(K_4(\tau)+\bar{F}_{4})\,{\rm sgn}(X_{X_0,Y_0}{(\tau)})\right)\right)\right] d\tau,
\end{array}
\end{equation}
where the functions of $\bar F_i$ for $i=1,2,3,4$  are the ones given in
\eqref{jj}. Note that $-a+b+\sqrt{\Delta}>0$, then
\[
{\rm sgn}\left(\dfrac{ -a+b+\sqrt{\Delta}}{2 b}X_{X_0,Y_0}{(\tau)}\right)={\rm sgn}\left(X_{X_0,Y_0}{(\tau)}\right),
\]
denoting by $K=1/(2\sqrt{\Delta}pT_{1})$, the function \eqref{S1} $\CG_1(X_0,Y_0)$ becomes
\[
\begin{array}{L}
-K \int_0^{pT_1}\sin\left(\omega_{1} \, \tau\right)(K_3(\tau)+\bar{F}_3) \left(a-b+\sqrt{\Delta}\right)+2 b (K_1(\tau)+\bar{F}_1)\,d\tau\\
-K \int_0^{pT_1}\sin\left(\omega_{1} \, \tau\right)\left((K_4(\tau)+\bar{F}_4) \left(a-b+\sqrt{\Delta}\right)+2 b (K_2(\tau)+\bar{F}_2)\right)\,{\rm sgn}(X_{X_0,Y_0}{(\tau)}) \,d\tau,
\end{array}
\]
and the function \eqref{S1a} $\CG_2(X_0,Y_0)$ becomes
\[
\begin{array}{L}
K \int_0^{pT_1}\cos\left(\omega_{1} \, \tau\right)(K_3(\tau)+\bar{F}_3) \left(a-b+\sqrt{D}\right)+2 b (K_1(\tau)+\bar{F}_1)\,d\tau\\
K \int_0^{pT_1}\cos\left(\omega_{1} \, \tau\right)\left((K_4(\tau)+\bar{F}_4) \left(a-b+\sqrt{D}\right)+2 b (K_2(\tau)+\bar{F}_2)\right)\,{\rm sgn}(X_{X_0,Y_0}{(\tau)}) \,d\tau.
\end{array}
\]

Then, by Theorem \ref{tt} we have that for every simple
zero $(X_0^*,Y_0^*)\in V$ of the system of non-linear and non-smooth functions
\begin{equation}\label{hh}
\CG_1(X_0,Y_0)=0 \quad,\quad \CG_2(X_0,Y_0)=0,
\end{equation}
we have a periodic solution $(X,Y,Z,W)(t,\e)$ of system
\eqref{d2} such that
\[
(X,Y,Z,W)(0,\e)\to (X_0^*,Y_0^*,0,0)\quad \mbox{as $\e\to 0$.}
\]
Note that system \eqref{hh} is equivalent to system \eqref{e4},
because both equations only differs in a non--zero multiplicative
constant.

\smallskip

Going back through the change of coordinates \eqref{nn} we get a
periodic solution $(x,y,z,w)(\tau,\e)$ of system \eqref{d3} such
that
\[
\left(
\begin{array}{c}
x(\tau,\e) \vspace{0.2cm}\\
y(\tau,\e) \vspace{0.2cm}\\
z(\tau,\e) \vspace{0.2cm}\\
w(\tau,\e)
\end{array}
\right)\to \left(
\begin{array}{c}
\dfrac{-a+b+\sqrt{\Delta}}{2\, b \,\omega_{1}} \left(X_0^* \cos \left(\omega_{1}\, \tau \right) + Y_0^* \sin
\left(\omega_{1}\, \tau \right)\right) \vspace{0.2cm}\\
\dfrac{-a+b+\sqrt{\Delta}}{2 b} \left(Y_0^* \cos \left(\omega_{1}\, \tau \right) - X_0^* \sin
\left(\omega_{1}\, \tau \right)\right) \vspace{0.2cm}\\
\dfrac{1}{\omega_{1}} \left(X_0^* \cos \left(\omega_{1}\, \tau \right) + Y_0^* \sin
\left(\omega_{1}\, \tau \right) \right) \vspace{0.2cm}\\
Y_0^* \cos \left(\omega_{1}\, \tau \right) - X_0^* \sin
\left(\omega_{1}\, \tau \right)\end{array}
\right)
\]
as $\e\to 0$.

\smallskip

Consequently we obtain a periodic solution $(\T_1,\T_2)(\tau,\e)$ of
system \eqref{pendulum2} such that
\[
(\T_1,\T_2)(\tau,\e)\to\left(
\begin{array}{c}
\dfrac{-a+b+\sqrt{\Delta}}{2\, b \,\omega_{1}} \left(X_0^* \cos \left(\omega_{1}\, \tau \right) + Y_0^* \sin
\left(\omega_{1}\, \tau \right)\right) \vspace{0.2cm}\\
\dfrac{1}{\omega_{1}} \left(X_0^* \cos \left(\omega_{1}\, \tau \right) + Y_0^* \sin
\left(\omega_{1}\, \tau \right) \right) \vspace{0.2cm}
\end{array}
\right)
\]
as $\e\to 0$. Hence Theorem \ref{t1} is proved.
\end{proof}

\smallskip

\begin{proof}[Proof of Theorem \ref{t1a}]
This proof is completely analogous to the proof of Theorem \ref{t1}.
\end{proof}

\section{Proofs of corollaries}\label{s4}
To obtain the expression of the functions given in \eqref{e3a} and \eqref{e3b} we have to study the changes of sign of the functions $X_{X_0,Y_0}{(\tau)}$ and $Z^{Z_0,W_0}{(\tau)}$ respectively for $t\in[0,pT_1]$ and $t\in[0,pT_2]$.

Note that $X_{X_0,Y_0}{(t_n)}=0$ for 
\[
t_n=\dfrac{1}{\omega_1}\left(\pi  n-\arctan\left(\dfrac{Y_0}{X_0}\right)\right).
\]

If $X_0Y_0<0$, then $t_n\in[0,pT_1]$ only for $n=0,1,\cdots,p+1$, and if $X_0Y_0>0$, then $t_n\in[0,pT_1]$ only for $n=1,2,\cdots,p+2$. We know that for all $t\in[t_n,t_{n+1}]$ the function $Y_{X_0,Y_0}{(\tau)}$ has the same sign and different sign for any $t\in[t_{n-1},t_n]$, thus the integral can be computed using the partitions $\{0,\,t_n,\,pT_1\,;n=0,1,\cdots,p+1\}$ and $\{0,\,t_n,\,pT_1\,;n=1,2,\cdots,p+2\}$ as the limits of integration respectively for $X_0Y_0<0$ and $X_0Y_0>0$.

The study of changes of the sign of the function $Z^{Z_0,W_0}{(\tau)}$ for $t\in[0,pT_2]$ and $Z_0W_0\neq0$ is completely analogous.

\begin{proof}[Proof of Corollary \ref{c1}]
Firstly, we have to check the {\it Crossing Hypothesis} (see Remark \ref{CH}) to the system \eqref{e1a}. 

Note that we have four different vector fields defined in four different regions (see Figure \ref{Regions}).

\smallskip

In the region $\CR_1=\{x>0 \textrm{ and } z>0\}$ we have
\[
\X_1=
\left(\begin{array}{l}
y\\
-ax+z+\e\dfrac{y+1}{\al^2}+\e^2R_1(t,x,y,z,w)\\
w\\
bx-bz+\e\dfrac{w+1}{\al^2}+\e^2R_2(t,x,y,z,w),
\end{array}\right).
\]

\smallskip

In the region $\CR_2=\{x<0 \textrm{ and } z>0\}$ we have
\[
\X_2=
\left(\begin{array}{l}
y\\
-ax+z+\e\dfrac{y-1}{\al^2}+\e^2R_1(t,x,y,z,w)\\
w\\
bx-bz+\e\dfrac{w+1}{\al^2}+\e^2R_2(t,x,y,z,w),
\end{array}\right).
\]
\smallskip

In the region $\CR_3=\{x<0 \textrm{ and } z<0\}$ we have
\[
\X_3=
\left(\begin{array}{l}
y\\
-ax+z+\e\dfrac{y-1}{\al^2}+\e^2R_1(t,x,y,z,w)\\
w\\
bx-bz+\e\dfrac{w-1}{\al^2}+\e^2R_2(t,x,y,z,w),
\end{array}\right).
\]
\smallskip

Finally, in the region $\CR_4=\{x>0 \textrm{ and } z<0\}$ we have
\[
\X_4=
\left(\begin{array}{l}
y\\
-ax+z+\e\dfrac{y+1}{\al^2}+\e^2R_1(t,x,y,z,w)\\
w\\
bx-bz+\e\dfrac{w-1}{\al^2}+\e^2R_2(t,x,y,z,w),
\end{array}\right).
\]

\begin{figure}[h]
\centering
\psfrag{B}{$z=0$}
\psfrag{D}{$x=0$}
\psfrag{U1}{$\X_1$}
\psfrag{U2}{$\X_2$}
\psfrag{U3}{$\X_3$}
\psfrag{U4}{$\X_4$}
\psfrag{g1}{$g_2^{-1}(0)$}
\psfrag{g2}{$g_1^{-1}(0)$}
\includegraphics[width=5cm]{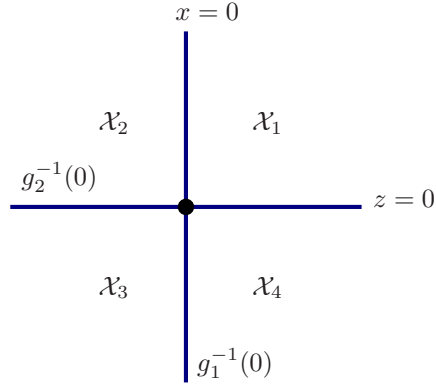}
\vskip 0cm \centerline{} \caption{\small \label{Regions} {\bf Four different vector fields.}}
\end{figure}

\smallskip

To study the types of the sets $\CM_{ij}$ (see Appendix A), we have to compute Lie derivative of the functions $g_1$ and $g_2$ with respect to the vector fields $\X_i$ for $i=1,2,3,4$, i.e.
\[
(\Lie_{\X_i})(g_j)(x,y,z,w)=\langle \nabla g_j,\X_i\rangle(x,y,z,w).
\]

Proceeding with these calculations we have
\[
\begin{array}{CL}
(\Lie_{\X_1})(g_1)(x,y,z,w)=(\Lie_{\X_2})(g_1)(x,y,z,w)=&y,\\

(\Lie_{\X_2})(g_2)(x,y,z,w)=(\Lie_{\X_3})(g_2)(x,y,z,w)=&w,\\

(\Lie_{\X_3})(g_1)(x,y,z,w)=(\Lie_{\X_4})(g_1)(x,y,z,w)=&y,\\

(\Lie_{\X_1})(g_2)(x,y,z,w)=(\Lie_{\X_4})(g_2)(x,y,z,w)=&w.
\end{array}
\]
Hence we can conclude that in the set
\[
\CT=\{(x,y,0,0)\}\bigcup\{(0,0,z,w)\},
\]
the flow is tangent to the discontinuous set, and in any other point the flow cross the set of discontinuity.

Using the coordinates defined in \eqref{nn}, we have that
\[
\CT=\left\{\left(X,Y,-\dfrac{\omega_2}{\omega_1}X,-Y\right)\right\}\bigcup\left\{\left(X,Y,\beta\dfrac{\omega_2}{\omega_1}X,\beta Y\right)\right\}.
\]
Observe that the periodic orbits given by Lemma \ref{L1} filling the planes $\{(X,Y,0,0)\}$ and $\{(0,0,X,Y)\}$, except the origin, do not reach the set $\CT$. Thus, for $|\e|>0$ sufficiently small, there exists a neighborhood of the planes $\{(X,Y,0,0)\}$ and $\{(0,0,X,Y)\}$ such that the orbits cross the set of discontinuity. In other words, the Crossing Hypothesis is satisfied.

Studying the changes of the sign of the function $X_{X_0,Y_0}{(\tau)}$ for $t\in\left[0,T_1\right]$ we conclude that the non-smooth functions \eqref{e3a} and \eqref{e3b} are given by
\[
\CF_1(X_0,Y_0)=
\left\{\begin{array}{LCC}
-\dfrac{2\sqrt{\Delta}\pi}{\omega_1}Y_0+\dfrac{4(a+b+\sqrt{\Delta})}{\omega_1\sqrt{1+\dfrac{Y_0^2}{X_0^2}}}&\textrm{if}& X_0>0,\\
-\dfrac{2\sqrt{\Delta}\pi}{\omega_1}Y_0-\dfrac{4(a+b+\sqrt{\Delta})}{\omega_1\sqrt{1+\dfrac{Y_0^2}{X_0^2}}}&\textrm{if}& X_0<0,\\
\end{array}\right.
\]

\[
\CF_2(X_0,Y_0)=
\left\{\begin{array}{LCC}
\dfrac{2\sqrt{\Delta}\pi}{\omega_1}Y_0-\dfrac{4(a+b+\sqrt{\Delta})}{\omega_1\sqrt{1+\dfrac{Y_0^2}{X_0^2}}}&\textrm{if}& X_0>0,\\
\dfrac{2\sqrt{\Delta}\pi}{\omega_1}Y_0+\dfrac{4(a+b+\sqrt{\Delta})}{\omega_1\sqrt{1+\dfrac{Y_0^2}{X_0^2}}}&\textrm{if}& X_0<0,\\
\end{array}\right.
\]

This system has all solutions inside a periodic orbit of the unperturbed systems passing through
\[
(X_0^*,Y_0^*)=\left(\dfrac{\sqrt{2}}{\sqrt{\Delta}\pi}(a+b+\sqrt{\Delta})\,,\,\dfrac{\sqrt{2}}{\sqrt{\Delta}\pi}(a+b+\sqrt{\Delta})\right).
\]
It is easy to check that this solution are
simple. So, by Theorem \ref{t1} we have one periodic solution
of the non-smooth perturbed double pendulum. This completes the proof of the
corollary.
\end{proof}

\begin{proof}[Proof of Corollary \ref{c2}]
This proof is completely analogous to the proof of Corollary \ref{c1}.
\end{proof}

\section*{Appendix A: Basic concepts on Filippov systems}\label{ap}

We say that a vector field $X:D\subset\R^n\rightarrow\R^n$ is {\it Piecewise Continuous} if the domain $D$ can be partitioned in a finite collection of connected, open and disjoint sets $D_i$, $i=1,\cdots,k$, such that, the vector field $X\big|_{\ov{D}_i}$ is continuous for $i=1,\cdots,k$.

\smallskip

We denote by $S_X\subset\p D_1\cup\cdots\cup\p D_k$ the set of points where the vector field $X$ is discontinuous. By assumptions, the set $S_X$ has measure zero.

\smallskip

If $\CM\subset S_X$ is a manifold of codimension one, then $\CM$ can be decomposed as the union of the closure of the regions (see Figure \ref{fig2}):
\[
\begin{array}{LL}
\Sigma^c=\left\{x\in\CM: (Xh)(Yh)(x)>0\right\};\\
\Sigma^e=\left\{x\in\CM: (Xh)(x)>0 \textrm{ e } (Yh)(x)<0\right\};\\
\Sigma^s=\left\{x\in\CM: (Xh)(x)<0 \textrm{ e } (Yh)(x)>0\right\}.\\
\end{array}
\]

\begin{figure}[h]
\centering
\psfrag{SC}{$\Sigma^c$}
\psfrag{SS}{$\Sigma^s$}
\psfrag{SE}{$\Sigma^e$}
\includegraphics[width=10cm]{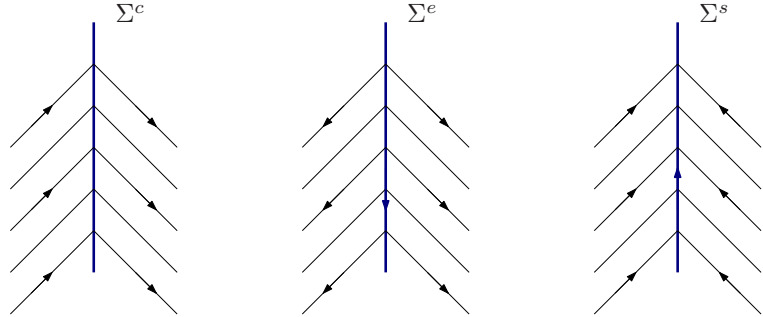}
\vskip 0cm \centerline{} \caption{\small \label{fig2} Crossing region $(\Sigma^c)$, escaping region $(\Sigma^e)$ and sliding region $(\Sigma^s)$.}
\end{figure}

\smallskip

For $p\in\Sigma^e\cup\Sigma^s$ we define the {\it Sliding Vector Field} as
\begin{equation}\label{cc}
Z_s(p)=\dfrac{1}{(Yh)(p)-(Xh)(p)}\left((Yh)(Xh)(p)-(Xh)(Yh)(p)\right).
\end{equation}

\smallskip

Consider the following equation
\begin{equation}\label{Ae1}
\dot{x}=X(x),
\end{equation}
where $X:D\subset\R^n\rightarrow\R^n$ is  a piecewise continuous vector field. The local solution of the equation \eqref{Ae1} passing through a point $p\in\CM$ is given by the Filippov convention: 

\begin{itemize}
\item[(i)] for $p\in\Sigma^c$ such that $(Xh)(p),(Yh)(p)>0$ and taking the origin of time at $p$, the trajectory is defined as $\f_Z(t,p)=\f_Y(t,p)$ for $t\in I_p\cap\{t<0\}$ and $\f_Z(t,p)=\f_X(t,p)$ for $t\in I_p\cap\{t>0\}$. For the case $(Xh)(p),(Yh)(p)<0$ the definition is the same reversing time;

\smallskip

\item[(i)] for $p\in\Sigma^e\cup\Sigma^s$ such that $Z_s(p)\neq 0$, $\f_Z(t,p)=\f_{Z_s}(t,p)$ for $t\in I_p\subset \R$.
\end{itemize}

\smallskip

Here $\f_W$ denotes the flow of a vector field $W$.

\smallskip

For more details about discontinuous differential equation see Filippov's book \cite{F}.

\section*{Appendix B: Basic results on averaging theory}\label{ap}

We present the basic result from the averaging
theory that we shall need for proving the main results of this
paper.

\smallskip

We consider the problem of the bifurcation of $T$--periodic
solutions from differential systems of the form
\begin{equation}\label{eq:4}
\dot {\bf x}(t)= G_0(t,{\bf x})+\e G_1(t,{\bf x})+\e^2 G_{2}(t,{\bf
x}, \e),
\end{equation}
with $\e=0$ to $\e\not= 0$ sufficiently small. Here the functions
$G_0,G_1: \R \times \Omega \to \R^n$ and $G_{2}:\R\times \Omega
\times (-\e_0,\e_0)\to \R^n$ are $\CC^2$ functions, $T$--periodic in
the first variable, and $\Omega $ is an open subset of $\R^n$. The
main assumption is that the unperturbed system
\begin{equation}\label{eq:5}
\dot {\bf x}(t)= G_0(t,{\bf x}),
\end{equation}
has a submanifold of periodic solutions. A solution of this problem
is given using the averaging theory.

\smallskip

Let ${\bf x}(t,{\bf z},\e)$ be the solution of the system
\eqref{eq:5} such that ${\bf x}(0,{\bf z},\e)= {\bf z}$. We write
the linearization of the unperturbed system along a periodic
solution ${\bf x}(t,{\bf z},0)$ as
\begin{equation}\label{eq:6}
\dot {\bf y}=D_{\bf x}{G_0}(t,{\bf x}(t,{\bf z},0)){\bf y}.
\end{equation}
In what follows we denote by $M_{\bf z}(t)$ some fundamental matrix
of the linear differential system \eqref{eq:6}, and by
$\xi:\R^k\times \R^{n-k}\to \R^k$ the projection of $\R^n$ onto its
first $k$ coordinates; i.e. $\xi(x_1,\ldots,x_n)= (x_1,\ldots,x_k)$.

\smallskip

We assume that there exists a $k$--dimensional submanifold
$\mathcal{Z}$ of $\Omega$ filled with $T$--periodic solutions of
\eqref{eq:5}. Then an answer to the problem of bifurcation of
$T$--periodic solutions from the periodic solutions contained in
$\mathcal{Z}$ for system \eqref{eq:4} is given in the following
result.

\begin{theorem}\label{tt}
Let $V$ be an open and bounded subset of $\R^k$, and let $\beta:
\cl(V)\to \R^{n-k}$ be a $\CC^2$ function. We assume that
\begin{itemize}
\item[(i)] $\mathcal{Z}=\left\{ {\bf z}_{\alpha}=\left( \alpha,
\beta(\alpha)\right),~~\alpha\in \cl(V) \right\}\subset \Omega$ and
that for each ${\bf z}_{\alpha}\in \mathcal{Z}$ the solution ${\bf
x}(t,{\bf z}_{\alpha})$ of \eqref{eq:5} is $T$--periodic;

\item[(ii)] for each ${\bf z}_{\alpha}\in \mathcal{Z}$ there
is a fundamental matrix $M_{{\bf z}_{\alpha}}(t)$ of \eqref{eq:6}
such that the matrix $M_{{\bf z}_{\alpha}}^{-1}(0)- M_{{\bf
z}_{\alpha}}^{-1}(T)$ has in the upper right corner the $k\times
(n-k)$ zero matrix, and in the lower right corner a $(n-k)\times
(n-k)$ matrix $\Delta_{\alpha}$ with $\det(\Delta_{\alpha})\neq 0$.
\end{itemize}
We consider the function $\CG:\cl(V) \to \R^k$
\begin{equation}\label{eq:7}
\CG(\alpha)=\xi\left( \dfrac{1}{T} \int _0^T M_{{\bf
z}_{\alpha}}^{-1}(t)G_1(t,{\bf x}(t,{\bf z}_{\alpha})) dt\right).
\end{equation}
If there exists $a\in V$ with $\CG(a)=0$ and $\displaystyle{\det
\left( \left({d\CG}/{d\alpha}\right)(a)\right)\neq 0}$, then there
is a $T$--periodic solution $\varphi (t,\e)$ of system \eqref{eq:4}
such that $\varphi(0,\e)\to {\bf z}_a$ as $\e\to 0$.
\end{theorem}

Theorem~\ref{tt} goes back to Malkin \cite{Ma} and Roseau \cite{Ro},
for a shorter proof see \cite{BFL}.

\section*{Acknowledgements}

The first author is partially supported by a MICIIN/FEDER grant
MTM2008--03437, an AGAUR grant number 2009SGR-0410, an ICREA
Academia and FP7--PEOPLE--2012--IRSES--316338. The second author is partially suported by the grant
FAPESP 2011/03896-0 The third author is partially supported by a
FAPESP--BRAZIL grant 2007/06896--5. The first and third authors are
also supported by the joint project CAPES-MECD grant
PHB-2009-0025-PC

\end{document}